\documentclass[10pt]{article}
\usepackage{amsfonts}
\usepackage{graphicx}
\usepackage{amsmath}
\usepackage{amsthm}
\usepackage{amssymb}
\usepackage{amscd}
\usepackage{graphicx}
\usepackage{microtype}
\usepackage{enumerate}
\usepackage{pinlabel}
\usepackage[top=1.3in,bottom=1.6in,left=1.3in,right=1.3in]{geometry}

\theoremstyle{definition}
\newtheorem{theorem}{Theorem}[section]

\newtheorem{corollary}[theorem]{Corollary}

\newtheorem{example}[theorem]{Example}

\newtheorem{fact}[theorem]{Fact}
\newtheorem{lemma}[theorem]{Lemma}
\newtheorem{notation}[theorem]{Notation}

\newtheorem{observation}[theorem]{Observation}

\newtheorem{proposition}[theorem]{Proposition}
\newtheorem{remark}[theorem]{Remark}

\newtheorem{question}[theorem]{Question}
\newtheorem{definition}[theorem]{Definition}
\newtheorem*{dimthm}{Theorem \ref{main cor}}

\mathchardef\ordinarycolon\mathcode`\:
\mathcode`\:=\string"8000
\begingroup \catcode`\:=\active
  \gdef:{\mathrel{\mathop\ordinarycolon}}
\endgroup

\newcommand{\R}{\mathbb{R}}

\newcommand{\Z}{\mathbb{Z}}
\newcommand{\N}{\mathbb{N}}
\newcommand{\Q}{\mathbb{Q}}

\newcommand{\T}{\mathbb{T}}

\newcommand{\boldhead}[1]{%
 {\bigskip \noindent \bfseries #1 \\ }}

\DeclareMathOperator{\Diff}{Diff}
\DeclareMathOperator{\Homeo}{Homeo}

\DeclareMathOperator{\interior}{int}

\DeclareMathOperator{\supp}{supp}
\DeclareMathOperator{\fix}{Fix}

\title{Homomorphisms between diffeomorphism groups}
\author{Kathryn Mann}
\date{}

\begin{document}

\maketitle

\begin{abstract}
For $r \geq 3$, $p \geq 2$, we classify all actions of the groups $\Diff_c^r(\R)$ and $\Diff^r_+(S^1)$ by $C^p$-diffeomorphisms on the line and on the circle.  This is the same as describing all nontrivial group homomorphisms between groups of compactly supported diffeomorphisms on 1-manifolds.   We show that all such actions have an elementary form, which we call \emph{topologically diagonal}.

As an application, we answer a question of Ghys in the 1-manifold case: if $M$ is any closed manifold, and $\Diff^\infty(M)_0$ injects into the diffeomorphism group of a 1-manifold, must $M$ be 1 dimensional?  We show that the answer is yes, even under more general conditions.  

Several lemmas on subgroups of diffeomorphism groups are of independent interest, including results on commuting subgroups and flows.  
\end{abstract}

\setcounter{section}{0}

\section{Introduction}

For a manifold $M$ (assumed connected and without boundary but not necessarily closed), let $\Diff_c^r(M)$ denote the group of compactly supported, orientation-preserving, $C^r$-diffeomorphisms of $M$ isotopic to the identity.  We study all nontrivial homomorphisms $\Phi: \Diff_c^r(M_1) \to \Diff_c^p(M_2)$ between two such diffeomorphism groups.  
Here are some easy examples to keep in mind.
\begin{example} \label{easy} \hspace{1pt}
\begin{enumerate}[a)]
\item If $M_1$ is an open submanifold of $M_2$,  there is a natural inclusion $\Diff_c^r(M_1) \hookrightarrow \Diff_c^r(M_2)$ by extending any compactly supported diffeomorphism of $M_1$ to be the identity on $M_2 \setminus M_1$.
\item A $C^r$-diffeomorphism $f: M_1 \to M_2$ induces an isomorphism $\Phi: \Diff_c^r(M_1) \to \Diff_c^r(M_2)$ defined by $\Phi(g) = fgf^{-1}$.   
\item Combining the two examples above, a diffeomorphism $f$ from $M_1$ to an open subset of $M_2$ can be used to define a natural injection $\Phi: \Diff_c^r(M_1) \hookrightarrow \Diff_c^r(M_2)$ by 
$$\Phi(g)(x) = \left\{ \begin{array}{rcl}  &f g f^{-1}(x) &  \mbox{if }x \in f(M_1) \\  &x & \mbox{otherwise} \end{array}\right.$$
If $f$ is instead $C^p$ for some $p < r$, then the image of $\Phi$ will lie in $\Diff_c^p(M_2)$.
\item Finally, there are some special cases where $M_1 \subset M_2$ is a non-open submanifold and diffeomorphisms of $M_1$ still extend to $M_2$ in a natural way.  This happens for example if $M_2 = M_1 \times N$ for a compact manifold $N$, or if $M_2$ is the unit tangent bundle of $M_1$.  
  
\end{enumerate}  
\end{example}

Our goal is to find conditions on $M_1$ and $M_2$ that guarantee the existence of nontrivial homomorphisms  $\Diff_c^r(M_1) \to \Diff_c^r(M_2)$,  and to describe as much as possible what these homomorphisms look like.  Note that even though all the examples of homomorphisms above are continuous, we make no assumptions on continuity. 

The analogous question for \emph{isomorphisms} rather than \emph{homomorphisms} between diffeomorphism groups is completely answered by a theorem of Filipkiewicz. It says that case b) of Example \ref{easy} is all that can occur.

\begin{theorem}[Filipkiewicz, \cite{Fi}] \label{fil}
Let $M_1$ and $M_2$ be smooth manifolds and suppose there is an isomorphism $\Phi: \Diff^r_c(M_1) \to \Diff^p_c(M_2)$. Then $M_1$ and $M_2$ are diffeomorphic.  In fact, more is true: $r = p$ and $\Phi$ is induced by a $C^r$-diffeomorphism $f: M_1 \to M_2$.  
\end{theorem}

Filipkiewicz's original statement was for isomorphisms between the identity components of the full diffeomorphism groups, but his theorem holds for the group of compactly supported diffeomorphisms as well.  More recently, Filipkiewicz's theorem has been generalized to isomorphisms between other groups of diffeomorphisms (e.g. symplectomorphisms, contact diffeomorphisms) by Banyaga, Rybicki, and others (see for example \cite{Ba2}, \cite{Ry}).   The spirit of all these results is the same: given a ``sufficiently large" class of diffeomorphisms $\mathcal{C}$, the existence of an isomorphism between $\mathcal{C}(M_1)$ and $\mathcal{C}(M_2)$ implies that $M_1$ and $M_2$ are diffeomorphic and the isomorphism is induced by a diffeomorphism.  

By contrast, we know almost nothing about homomorphisms between diffeomorphism groups.  As an example of our ignorance, consider the following question of Ghys.  Here $\Diff^{\infty}(M)_0$ denotes the identity component of the group of $C^\infty$ diffeomorphisms of $M$.

\begin{question}[Ghys, \cite{Gh}] \label{ghq}
Let $M_1$ and $M_2$ be closed manifolds, and suppose that there is an (injective) homomorphism $\Diff^{\infty}(M_1)_0 \to \Diff^{\infty}(M_2)_0$.  Does it follow that $\dim(M_1) \leq \dim(M_2)$?
\end{question}

In fact, $\Diff^{\infty}(M)_0$ is a simple group, so any nontrivial homomorphism is necessarily injective.  That $\Diff^{\infty}(M)_0$ is simple follows from a deep result due to Mather and Thurston, nontrivial even in the case of 1-manifolds (!).  Simplicity holds for a larger class of groups as well; for any connected manifold $M$ and any $r \neq \dim(M) + 1$, the group $\Diff^r_c(M)_0$ is simple (\cite{Ma1}, \cite{Ma2}).  Thus, it makes sense to ask a more general version of Question \ref{ghq} for noncompact manifolds, replacing $\Diff^\infty(M_i)_0$ with $\Diff^r_c(M_i)$.  

Ghys' question appeared in print in 1991, and to our knowledge, this paper is the first to give even a partial answer in a special case -- this will be a corollary of our Theorem \ref{main cor} below.  
However, since publication of this paper S. Hurtado and S. Matsumoto have made very promising further progress.  In \cite{Hu}, Hurtado gives a complete answer to Ghys' question (in its original form) in the affirmative.  His results use the assumption of $C^\infty$ smoothness of diffeomorphisms in an essential way.   In \cite{Mat}, Matsumoto answers our more general version of Ghys' question under the assumption that $\dim(M_2) = 1$ and assuming only $C^r$ regularity of diffeomorphisms ($r \geq 2$), without the classification of homomorphisms that we give in Theorem \ref{main cor}.   See also \cite{Mat2} for a special case treating homeomorphisms.

\boldhead{Our results}
We study the special case of homomorphisms $\Phi: \Diff_c^r(M_1) \to \Diff_c^p(M_2)$ when $M_2$ is a 1-manifold.  The reader will see the difficulty of the problem and some of the richness in examples already apparent at the 1-dimensional level.  However, we are able to give a complete answer to the question of which manifolds $M_1$ admit nontrivial homomorphisms $\Phi: \Diff_c^r(M_1) \to \Diff_c^r(M_2)$ or 
$\Phi: \Diff^r(M_1)_0 \to \Diff_c^r(M_2)$ for some 1-manifold  $M_2$, and describe precisely what all such homomorphisms look like.  
Essentially, $M_1$ must be 1-dimensional, and $\Phi$ described by a slight generalization of the obvious embedding of diffeomorphism groups given in Example \ref{easy} c).  This generalization is constructed by taking multiple embeddings $M_1 \hookrightarrow M_2$ and having diffeomorphisms of $M_1$ act ``diagonally" on $M_2$.  We call this a \emph{topologically diagonal embedding}.  

\begin{definition}[\textbf{Topologically diagonal embedding}]
Let $M_1$ and $M_2$ be manifolds, and suppose $\{ f_1, f_2, f_3, ...\}$ is a collection (finite or infinite) of $C^r$ embeddings $M_1 \to M_2$ whose images are pairwise disjoint and contained in some compact subset of $M_2$.  Define $\Phi: \Diff_c^r(M_1) \to \Diff_c^r(M_2)$ by 
$$\Phi(g)(x) = \left\{ \begin{array}{rcl}  &f_i g f_i^{-1}(x) &  \mbox{if }x \in f_i(M_1) \\  &x & \mbox{otherwise} \end{array}\right.$$
A $C^r$-\emph{topologically diagonal embedding} is a homomorphism $\Phi: \Diff_c^r(M_1) \to \Diff_c^r(M_2)$ obtained by this construction. 

A \emph{topologically diagonal embedding} is an embedding $\Phi: \Diff_c^r(M_1) \to \Diff_c^p(M_2)$ obtained by the same construction when the embeddings $f_i$ are only required to be continuous. 
\end{definition}

\noindent Our main theorem says that these define basically all homomorphisms between 1-manifold diffeomorphism groups.

\begin{theorem} \label{main}
Let $r \geq 3$, $p \geq 2$ and let $M_1$ and $M_2$ be 1-manifolds.  Every homomorphism $\Phi: \Diff_c^r(M_1) \to \Diff^p_c(M_2)$ is topologically diagonal.  If $r=p$, then $\Phi$ is $C^r$-topologically diagonal. 
\end{theorem}

\noindent We also have a stronger version.
  
\begin{theorem} \label{main cor}
Let $r \geq 3$, $p \geq 2$, let $M_1$ be any manifold and $M_2$ a 1-manifold.  Suppose that $\Phi: \Diff_c^r(M_1) \to \Diff_c^p(M_2)$ is an injective homomorphism.  Then $\dim(M_1) = 1$ and $\Phi$ is topologically diagonal.  If $r=p$, then $\Phi$ is $C^r$-topologically diagonal. 
\end{theorem}
\noindent In particular, this answers the 1-manifold case of Ghys' Question (\ref{ghq}) in the affirmative.  

\boldhead{The algebra--topology link}
There are two points to note regarding Theorems \ref{main} and \ref{main cor}.  First, our results parallel that of Filipkiewicz by establishing a relationship between the topology of a manifold and the algebraic structure of its group of diffeomorphisms.  In fact, our proof strategy will be to gradually pin down this relationship as closely as possible.  

Second, our results say that the richness of the algebraic structure of diffeomorphism groups \emph{forces} continuity: group homomorphisms $\Phi: \Diff_c^r(M_1) \to \Diff^p_c(M_2)$, where $M_2$ is a 1-manifold, are \emph{continuous} (assuming appropriate constraints on $p$ and $r$).  This need not be true if $\Diff_c^r(M_1)$ is replaced with another topological group!  Consider the following example.

\begin{example}[\textbf{A non-continuous, injective group homomorphism} $\R \to \Diff^r_c(M)$] \label{noncont}
Let $M$ be a manifold, and let $\alpha: \R \to \R$ be a non-continuous, injective additive group homomorphism.  Such a homomorphism may be constructed by permuting the elements of a basis for $\R$ over $\Q$ and extending linearly to $\R$.  Then for any compactly supported $C^r$ flow $\psi^t$ on $M$, the map $t \mapsto \psi^{\alpha(t)}$ is a non-continuous, injective homomorphism from $\R$ to $\Diff^r_c(M)$.   

For an even more pathological example, consider a second flow $\phi^t$ with support disjoint from $\psi^t$.  Then 
$$ t \mapsto \left\{ \begin{array}{rcl} \psi^{t} & \mbox{for} & t \in \Q 
\\ \phi^{\alpha(t)} & \mbox{for} & t \notin \Q \\ 
\end{array}\right.$$
also gives a non-continuous, injective homomorphism from $\R$ to $\Diff^r_c(M)$.  
\end{example}

A key step in our proof will be showing that this example cannot occur within the context of diffeomorphism groups:  if $\Phi: \Diff_c^r(M_1) \to \Diff^p_c(M_2)$ is a homomorphism between 1-manifold diffeomorphism groups, then any topological $\R$-subgroup of $\Diff_c^r(M_1)$ maps continuously into $\Diff_c^p(M_2)$.

\boldhead{Outline}
In Section \ref{comm} we establish a preliminarily algebraic--topological relationship using commuting subgroups.  This is done by developing the elementary theory of 1-manifold diffeomorphism groups stemming from Kopell's Lemma and H\"older's Theorem.  

In Section \ref{cont} we attack continuity.  Using the tools of Section \ref{comm}, we show that $\R$-subgroups of diffeomorphisms of a 1-manifold (flows) behave well under homomorphisms of diffeomorphism groups.  This rules out the kind of behavior in Example \ref{noncont}.

We prove Theorem \ref{main} in Section \ref{mainpf}, strengthening the algebraic--topological relationship from Section \ref{comm} and using continuity of $\R$-subgroups to construct continuous maps $f_i$ from $M_1$ to $M_2$ that define a topologically diagonal embedding.   

In Section \ref{ghys} we then compare 1-manifold diffeomorphism groups to $n$-manifold diffeomorphism groups and show that $\dim(M)=1$ can be detected as an algebraic invariant of a diffeomorphism group.  This is used to prove Theorem \ref{main cor}, answering the 1-manifold case of Ghys' question.  

Finally, Section \ref{SSS} concludes with some remarks on the necessity of our hypotheses (e.g. compact supports) and further questions.

\bigskip
\noindent The author wishes to thank Benson Farb, Amie Wilkinson and Bena Tshishiku for many helpful conversations regarding this work, and John Franks and Emmanuel Milton for their comments and careful reading.   Thanks also to the anonymous referee for suggestions to correct and clarify work in Section \ref{comm}.


\section{Commuting groups of diffeomorphisms} \label{comm}

\textbf{Algebraic--topological associations}\\
Our first goal is to associate the algebraic data of a group of diffeomorphisms of a manifold to the topological data of the manifold.  One example of the kind of association we have in mind appears in Filipkiewicz's proof of Theorem \ref{fil}.  Given an isomorphism between the diffeomorphism groups of two manifolds, Filipkiewicz builds a map between the manifolds by associating points of each manifold with their stabilizers -- the subgroups $\{g \in \Diff_c^r(M) : g(x) = x \}$.  Point stabilizers are maximal subgroups of $\Diff_c^r(M)$, so this gives a loose association between an algebraic property (maximality) and a topological object (point).  This association is eventually used to build a point-to-point map out of the algebraic data of a group isomorphism.  

However, Filipkiewicz's point/maximal subgroup association won't work for our purposes.  The main problem is that the maximality isn't preserved under group homomorphisms.  Whereas an \emph{isomorphism} $\Phi: \Diff_c^r(M_1) \to \Diff_c^r(M_2)$ maps maximal subgroups of $\Diff_c^r(M_1)$ to maximal subgroups of $\Diff_c^r(M_2)$, a \textit{homomorphism} may not.  To prove that group homomorphisms $\Diff_c^r(M_1) \to \Diff_c^r(M_2)$ have nice topological properties, we need to make use of an algebraic--topological correspondence robust under homomorphisms.  We will look at commuting, rather than maximal, subgroups. Our analysis starts with a trivial observation:

\begin{observation} Let $M$ be a manifold, $N \subset M$ a submanifold, and $H$ the subgroup of $\Diff^r(M)$ that fixes $N$ pointwise.  Then $H$ commutes with the subgroup of diffeomorphisms supported on $N$.  In other words, the centralizer of $H$ in $\Diff^r(M)$ is ``large".  
\end{observation}

It turns out that when $M$ is one-dimensional, this property is \emph{characteristic} of subgroups of $\Diff^r_c(M)$ that fix a submanifold.  Precisely, we have the following proposition.

\begin{proposition}  \label{simpler prop}
If $G$ is a nonabelian subgroup of $\Diff^r_c(\R)$ or $\Diff^r(I)$ and $G$ has nonabelian centralizer, then there is an open interval contained in $\fix(G)$.  
\end{proposition} 
 
Here, and in the sequel, $\fix(G)$ denotes the set $\{x \in M : g(x) = x \text{ for all } g \in G\}$.  Similarly, for a single diffeomorphism $g$, we set $\fix(g): =  \{x \in M : g(x) = x\}$.  

A similar statement to Proposition \ref{simpler prop} holds for $S^1$; it is stated as Proposition \ref{keypropS1} below.  But before we consider the $S^1$ case, let us first develop some of the necessary background for the proof of Proposition \ref{simpler prop}.

\boldhead{H\"older - Kopell Theory}
To prove Proposition \ref{simpler prop} and its counterpart for diffeomorphisms of $S^1$, we extend the theory of commuting diffeomorphisms of 1-manifolds that stems from two classical theorems (Theorems \ref{kopell} and \ref{holder} below) due to Otto H\"older and Nancy Kopell.  This theory for 1-manifolds is already surprisingly rich -- one can use Theorems \ref{kopell} and \ref{holder} to prove, for instance, that all nilpotent groups of $\Diff^2_+(S^1)$ are abelian, and that certain dynamically characterized subgroups of $\Diff^2_+(\R)$ are all conjugate into the affine group.  Details and some background can be found in \cite{FF} and in chapter 4 of \cite{Na}.  Here we restrict our attention to results needed to prove Proposition \ref{simpler prop}.  

\begin{theorem}[Kopell's Lemma, \cite{Ko}]  \label{kopell}
Let $g$ and $h$ be commuting $C^2$ diffeomorphisms of the half-open interval $I:= [0,1)$.  If $g$ is fixed point-free on the interior of $I$ and $h$ is not the identity, then $h$ is also fixed point free on $\interior(I)$.  
\end{theorem}

\begin{theorem} [H\"older, \cite{Ho}]  \label{holder}
Let $G$ be a group of orientation-preserving homeomorphisms of $S^1$ or of the open interval.  If $G$ acts freely (i.e. no element other than the identity has a fixed point), then $G$ is abelian.  
\end{theorem}  

Both of these theorems admit elementary proofs.  A nice proof of Theorem \ref{kopell} is given in Section 4.1.1 of \cite{Na}, and one of  Theorem \ref{holder} in \cite{FF}.  
The result most useful to us is the following combination of H\"older's Theorem and Kopell's Lemma. 

\begin{corollary}[\textbf{the H\"older-Kopell Lemma}] \label{holderkopell}
Let $g \in \Diff^2_+(I)$.  If $g$ has no interior fixed points, then the centralizer of $g$ in $\Diff^2_+(I)$ is abelian.  
\end{corollary} 

\begin{proof} Each element $h$ in the centralizer $C(g)$ commutes with $g$, so by Kopell's Lemma $h$ is either the identity or acts without fixed points on $\interior(I)$.  Thus, $C(g)$ acts freely and so by H\"older's Theorem $C(g)$ is abelian.  

\end{proof} 

\begin{remark} 
Corollary \ref{holderkopell} also follows from a more general theorem of Szekeres (Theorem \ref{szek} in Section \ref{cont}).  We don't need the generality of Szekeres' theorem here, and as it is a deeper result than H\"older's theorem and Kopell's lemma we postpone its introduction to Section \ref{cont}.  
\end{remark}

We now make use of Corollary \ref{holderkopell} to prove Proposition \ref{simpler prop}.  We note that this Proposition can be easily deduced from Lemma 3.2 and Theorem 3.5 of \cite{FF1} (which also proves some similar results for $S^1$), but we give a self-contained proof here for completeness.

\begin{proof}[Proof of Proposition \ref{simpler prop}]

\noindent We will repeatedly use the following elementary fact: 
\begin{fact} \label{easyfact}
$\fix(hgh^{-1}) = h \fix(g)$.  In particular, if $g$ and $h$ commute, then $\fix(g)$ is $h$-invariant.  
\end{fact}

Let $G$ be a nonabelian subgroup of $\Diff^r_+(I)$ with $r \geq 2$.   Let $H$ be the centralizer of $G$ in $\Diff^r_+(I)$, and assume that $H$ is nonabelian.  

First, note that each element $g \in G$ must have a fixed point in $(0,1)$, for if $g$ is fixed point free, then Corollary \ref{holderkopell} says that $C(g)$ is abelian, but $H$ (which we assumed to be non-abelian) is a subset of $C(g)$.  

We claim that in fact every endpoint of each component of $I \setminus \fix(g)$ is fixed by each $h \in H$.  To see this, let $h \in H$.  If some point $x \in \fix(g)$ is not fixed by $h$, then it lies in some connected component $J$ of $I \setminus \fix(h)$.  Since $g$ and $h$ commute, $\fix(h)$ is $g$-invariant. In particular $g$ permutes the connected components of $I \setminus \fix(h)$.  Since $x \in J$ is fixed by $g$, the interval $J$ must be $g$-invariant.  Apply the H\"older-Kopell lemma to $g|_{\bar{J}}$.  Since $h$ has no fixed points in $J$, but  $g$ does, we conclude that $g|_{\bar{J}} = id$.  Thus, the only fixed points of $g$ not necessarily fixed by $h$ are interior points of $\fix(g)$.  

Since the endpoints of each component of $I \setminus \fix(g)$ are fixed by each $h \in H$, each component of $I \setminus \fix(g)$ is $H$-invariant.  Consider again a particular connected component $J$ of $I \setminus \fix(g)$.  Kopell's lemma implies that $H|_{\bar{J}}$ is abelian.  Thus, if we let $\supp(g)$ denote the closure of $I \setminus \fix(g)$, we know that $H$ is abelian when restricted to $\bigcup_{g \in G} \supp(g)$, and hence on the closure of $\bigcup_{g \in G} \supp(g)$.  It follows that either there is some open set fixed by each $g \in G$ or $H$ is abelian everywhere.   

The proof for $G \subset \Diff_c^r(\R)$ follows quickly from the following observation 

\begin{observation} \label{embed observation}
There is an embedding $\iota: \Diff_c^r(\R) \to \Diff^r(I)$ whose image is the set of diffeomorphisms of $I$ supported away from the endpoints of $I$.  In particular, for $x \in \interior(I)$ we may define $\iota(g)(x)$ by $fg(f^{-1}(x))$ where $f: \R \to \interior(I)$ is a $C^\infty$ diffeomorphism.  
\end{observation} 

Hence, if $G \subset \Diff_c^r(\R)$ has nonabelian centralizer, then so does $\iota(G)$ as in Observation \ref{embed observation}.  Our proof above shows that $\fix(\iota(G))$ contains an open interval, hence $\fix(G)$ does as well.  

\end{proof}

 \begin{figure*}
   \labellist 
  \small\hair 2pt
   \pinlabel $\fix(G)$ at 130 70 
   \pinlabel $J$ at 245 75
   \pinlabel $\fix(H)$ at 200 30 
   \endlabellist
  \centerline{
    \mbox{\includegraphics[width=4in]{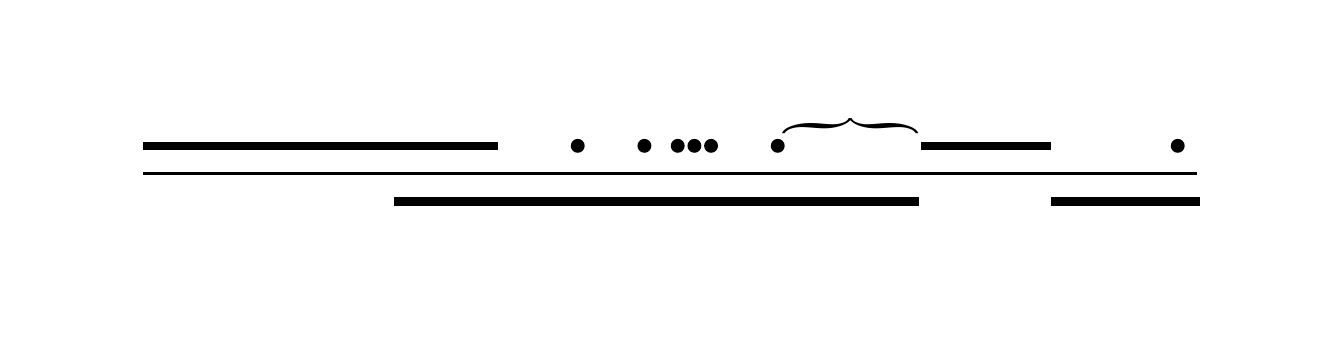}}}
 \caption{Fix sets of $G$ and $H$ indicated above and below the line.}
\centerline{If $H$ has a fixed point in $J$ and is not abelian on $J$, then $J \subset \fix(H)$.}
  \label{fixGfixH}
  \end{figure*}
  
A schematic picture of what fixed sets of nonabelian commuting $G$ and $H$ may look like is given in Figure \ref{fixGfixH}.    In particular we note the following corollary of the proof.  

\begin{corollary} \label{simpler prop cor} 
Let $G \subset \Diff^r(I)$ be as in proposition \ref{simpler prop}, and let $H$ be the centralizer of $G$.  Then $G$ and $H$ have a common fixed point in $\interior(I)$.  Similarly, if $G\subset \Diff_c^r(\R)$ and $H$ is the centralizer of $G$, then $G$ and $H$ have a common fixed point.  
\end{corollary}
\begin{proof}
We know that $\fix(G)$ contains a nontrivial interval, and the proof of proposition \ref{simpler prop} implies that each endpoint of this interval is fixed by all $h \in H$.  
\end{proof}

As another immediate consequence of Proposition \ref{simpler prop}, we have
\begin{corollary} \label{perfectcor}
Let $r \geq 2 $ and let $G$ and $H$ be commuting, nonabelian subgroups of $\Diff^r_+(I)$ or $\Diff^r_c(\R)$. 
If $H$ is perfect, then $H$ acts trivially on the complement of $\fix(G)$. 
\end{corollary}

\begin{proof} 
We showed above that for any $g \in G$, the restriction of its centralizer to $\R \setminus \fix(g)$ is abelian.  Since $H$ is contained in the centralizer, its restriction to $\R \setminus \fix(g)$ is abelian as well.  If $H$ is perfect, this restriction must be trivial. 
\end{proof}
Since the groups $\Diff_c^r(M)$ and $\Diff^r(M)_0$ are perfect as long as $r \neq \dim(M) + 1$, we will be able to apply Corollary \ref{perfectcor} in the sequel.

We would like to have an analog of Proposition \ref{simpler prop} for diffeomorphisms of the circle.  To do this, we develop a version of the H\"older-Kopell Lemma for $S^1$.

\boldhead{Circle diffeomorphisms} A major tool for studying homeomorphisms (or diffeomorphisms) of $S^1$ is the \emph{rotation number}.  This is an $\R/\Z$-valued number, well defined for each $g \in \Homeo_+(S^1)$.  If we consider $S^1$ as $\R/\Z$, the rotation number of an element $g$ is given by $\lim \limits_{n \to \infty} \frac{g^n(0)}{n}$, taken mod $\Z$.   A good exposition of the basic theory can be found in \cite{Na}.  We need only three facts here. 
\begin{enumerate}
\item A homeomorphism has rotation number 0 if and only if it has a fixed point.  
\item Homeomorphisms with rational rotation number of the form $m/k$ are precisely those with periodic points of period $k$. 
\item (Denjoy's theorem) If $g$ is a $C^2$ diffeomorphism of the circle with irrational rotation number $\theta$, then $g$ is topologically conjugate to rotation by $\theta$. 
\end{enumerate}

Using Fact 2, we see that if $g$ is a diffeomorphism with rational rotation number, then some iterate $g^k$ has a fixed point, hence can be viewed as a diffeomorphism of the interval.  This puts us in a situation where Kopell's Lemma applies.  If instead $g$ has irrational rotation number, Fact 3 tells us that $g$ is conjugate to an irrational rotation.  Since any homeomorphism that commutes with an irrational rotation must also be a rotation (see \cite{Na}), we conclude:

\begin{corollary} \label{hks1} Let $g \in \Diff^2_+(S^1)$.  Then there is either some $k \in \N$ such that $g^k$ has a fixed point, or the centralizer of $g$ in $\Diff^2_+(S^1)$ is abelian.  
\end{corollary} 

This is the analogue of the H\"older-Kopell Lemma for $S^1$.  Now we state and prove our analog of Proposition \ref{simpler prop}.

\begin{proposition} \label{keypropS1}
Let $r \geq 2$ and let $G \subset \Diff^r_+(S^1)$ be nonabelian.  Let $H$ be contained in the centralizer of $G$.  Then either 
\begin{enumerate}[a)]
\item (H is abelian after powers) For each pair of elements $h_1, h_2 \in H$, there are integers $k_1$ and $k_2$ such that $h_1^{k_1}$ and $h_2^{k_2}$ commute, or 
\item (G fixes an interval after powers) There is a subinterval $J \subset S^1$ and for each $g \in G$ an integer $k_g$ such that $J$ is fixed pointwise by $g^{k_g}$.
\end{enumerate}
\end{proposition}

\begin{proof}
Let $G$ and $H$ be as in the statement of Proposition \ref{keypropS1}.  
Let $g \in G$.  Since $H$ is nonabelian and centralizes $g$, Corollary \ref{hks1} implies that $g$ has rational rotation number.  Thus, after replacing $g$ with $g^{k_g}$ if necessary we can assume that $g$ has a fixed point.  We may also assume that $g$ was chosen so that $g^{k_g} \neq 1$, for if \emph{every} element of $G$ has finite order, then case b) holds. (In fact, in this case one can use H\"older's Theorem to show that $G$ is abelian).  

Now let $h_1$ and $h_2$ be any two elements of $H$.  Since $H$ and $G$ commute, Corollary \ref{hks1} implies again that $h_1$ and $h_2$ have rational rotation numbers, so there are integers $k_1$ and $k_2$ such that $h_1^{k_1}$ and $h_2^{k_2}$ have fixed points.  The same argument as in Proposition \ref{simpler prop} shows that endpoints of components of $S^1 \setminus \fix(g^{k_g})$ are fixed by each $h_i^{k_i}$.  Cutting $S^1$ at any one of the points of $\fix(g^{k_g})$ reduces the action of $g^{k_g}$, $h_1^{k_1}$ and $h_2^{k_2}$ to an action on the interval, and the proof of Proposition \ref{simpler prop} again shows that $h_1^{k_1}$ and $h_2^{k_2}$ must commute when restricted to $S^1 \setminus \fix(g^{k_g})$.  

Thus, if $h_1^{k_1}$ and $h_2^{k_2}$ do not commute everywhere, then there is some open set $J \subset S^1$ that is fixed pointwise by some iterate $g^{k_g}$ for every $g \in G$.   This is exactly what we wanted to prove.  

\end{proof}

\begin{remark} \label{rot0}
Note that the collection of elements $g^{k_g}$ that fix $J$ in case b) of Proposition \ref{keypropS1} are precisely the elements of $G$ with rotation number zero.  
In general, a collection of elements of $\Diff_+^2(S^1)$ that all have rotation number zero need not be a group, since the composition of two diffeomorphisms with fixed points need not have a fixed point.  But in this case the set of elements with rotation number zero it is equal to the subgroup of diffeomorphisms that fix $J$ pointwise.  
\end{remark}

We remark also that Proposition \ref{keypropS1} fails without the ``after powers" hypothesis -- there exist examples of two nonabelian, commuting subgroups of $S^1$, neither of which fixes any interval pointwise.  However, it would be interesting to know whether the ``after powers" condition could be replaced by ``virtually" (i.e. after taking a finite index subgroup).

\section{Flows and continuity} \label{cont}

Example \ref{noncont} in the introduction shows that not all $\R$-actions on a manifold are continuous.  Our goal in this section is to see (roughly) that whenever $M_1$ and $M_2$ are 1-manifolds, and $\Diff^r_c(M_1)$ includes in $\Diff^p_c(M_2)$, then $\R$ subgroups of $\Diff^r_c(M_1)$ \emph{do} act continuously -- at least on part of $M_2$.  

\begin{definition} By a \emph{flow} on a manifold $M$, we mean a continuous family of $C^r$-diffeomorphisms $\phi(t, x): \R \times I \to I$; in other words, an $\R$-subgroup of $\Diff^r(M)$ or $\Diff^r_c(M)$.  We denote $\phi(t, x)$ by $\phi^t(x)$, and by $\{\phi^t\}$ we mean the full $\R$-subgroup $\{\phi^t \in G : \, t\in \R\}$.  
\end{definition}

We will use the tools we developed in Section \ref{comm}, as well as the following further result in the spirit of H\"older-Kopell theory.  Keeping our notation from Section \ref{comm}, let $I:=[0,1)$.  We will also use $C_G(H)$ to denote the centralizer of a subgroup $H$ in a group $G$ and $C_G(g)$ for the centralizer of $g \in G$.  

\begin{theorem}[Szekeres, see 4.1.11 in \cite{Na}] \label{szek}
Let $r \geq 2$ and let $\psi \subset \Diff^r_+(I)$ be a diffeomorphism with no interior fixed points.  Then there is a unique flow $\psi^t$ of class $C^{r-1}$ on $(0,1)$ and $C^1$ on $I$ such that $\psi = \psi^1$ and the centralizer of $\psi$ in $\Diff^1_+(I)$ is $\{\psi^t\}$.  
\end{theorem}
Compare this theorem with the H\"older-Kopell Lemma (Corollary \ref{holderkopell}).  Under the same hypotheses, the H\"older-Kopell Lemma concludes that the centralizer of $\psi$ in $\Diff^r_+(I)$ is abelian.  Using Szekeres, we can conclude that $C_{\Diff^r_+(I)}(\psi)$ is a subgroup of a $C^1$ flow on $I$.

\boldhead{Flows on the interval}
Though our ultimate goal is a theorem about $\Diff^r_+(S^1)$ and $\Diff^r_c(\R)$, we first prove a result about flows in $\Diff^r_+(I)$.  This allows us to use Szekeres' theorem in a straightforward way, and to make the simplifying assumption that our flows have full support.   

\begin{proposition} \label{flows}
Let $p, r \geq 2$ and let $\Phi: \Diff^r_+(I) \to \Diff^p_+(I)$ be an injective
 group homomorphism.  Let $\{\phi^t\} \subset \Diff^r_+(I)$ be a flow that is fixed point free on $(0,1)$, and suppose that there is some $s$ such that $\Phi(\phi^s)$ is fixed point free on $(0,1)$.  
Then $\Phi(\{\phi^t\})$ is continuous in $t$.  Moreover, there is a flow $\{\psi^t\}$ such that $\Phi(\phi^t) = \psi^{\lambda t}$ for some $\lambda \in \R$.  
\end{proposition}

The proof of Proposition \ref{flows} has several steps.  We use Szekeres to find a candidate flow $\{\psi^t\}$ and reduce the question to continuity of a group homomorphism $\R \to \R$.  We next establish a simple criterion for continuity, and then relate this criterion to fixed sets and commutators and apply our work from Section \ref{comm}.  
Our main results of this section (Propositions \ref{rflow} and \ref{sflow}, stated later) are essentially built from this proof with the assumption of full supports removed.

\begin{proof}[Proof of Proposition \ref{flows}]
To simplify notation let $G = \Diff^r_+(I)$.   
Suppose $\Phi: G \to G$ is injective.  Let $\{\phi^t\} \subset G$ be a flow with no fixed points in $(0,1)$ and assume that $\psi := \Phi(\phi^s)$ has no fixed points in $(0,1)$.  Since $\{\phi^t\}$ is contained in  $C_G(\phi^1)$, its image $\{\Phi(\phi^t): t \in \R\}$ is contained in $C_{\Phi(G)}(\psi)$.  By Szekeres' theorem, there is a $C^1$ flow $\{\psi^t\}$ on $I$ such that $\{\Phi(\phi^t)\} \subset \{\psi^t\}$.  Moreover, that $\Phi$ is a group homomorphism implies that $\Phi(\phi^t) = \psi^{\alpha(t)}$ for some additive group homomorphism $\alpha: \R \to \R$.  Thus, we now only need to show that $\alpha$ is continuous.  

We can detect continuity in a very simple way using intersections of open sets.  Essentially, $\alpha$ being continuous just means that if $t$ is small enough then $\phi^{\alpha(t)}$ won't push an open set off of itself.  Formally, we have

\begin{lemma}[\textbf{Continuity criterion}]  \label{detect}
Let $\{\psi^t\}$ be a flow on $(0,1)$ such that $\psi^t$ is fixed point free for some (hence all) $t \neq 0$.  Let $\alpha: \R \to \R$ be a group homomorphism, and let $U \subset (0,1)$ be an open set that is bounded away from $0$ and $1$.  Then $\alpha$ is continuous if and only if there is some $\delta > 0$ such that $(\psi^{\alpha(t)}(U)) \cap U \neq \emptyset$ for all $t < \delta$.  
\end{lemma}

\begin{proof}
One direction is almost immediate: if $\alpha$ is a continuous $\R \to \R$ homomorphism, then $\alpha(t) = \lambda t$ for some $\lambda \in \R$ (easy exercise).  Given $\delta$ sufficiently small and $U$ open, for any $t < \delta$ we have $\psi^{\lambda t}(U) \cap U \neq \emptyset$.

For the converse, suppose that $\alpha$ is not continuous.  In this case, using the fact that $\alpha$ is an additive group homomorphism it is elementary to show that for any $T > 0$ there are arbitrarily small $t \in \R$ such that $|\alpha(t)| > T$.   Suppose $U$ is bounded away from $0$ and from $1$.  Since $\psi^t$ has no fixed points in $(0,1)$, for $T$ large enough, $\psi^{s}(U)$ will either be contained in a small neighborhood of 0 or a small neighborhood of 1 whenever $|s| > T$.  We can ensure this neighborhood is disjoint from $U$ by taking $T$ large.  In particular, when $|\alpha(t)|>T$ we will have $\psi^{\alpha(t)}(U) \cap U = \emptyset$.  
\end{proof}

The next step in the proof of Proposition \ref{flows} is to detect intersections of sets algebraically (i.e. through the algebraic structure of the diffeomorphism group), using the results of Section \ref{comm}.   The following notation will be useful.

\begin{notation}  For a group $G$ of diffeomorphisms of a manifold $M$ and a subset $U \subset M$, let 
$$G^U := \{g \in G : g(x) = x \text{ for all } x \in U\}.$$  
\end{notation}

The link between set intersections, subgroups, and centralizers comes from an elementary corollary of Proposition \ref{simpler prop}.
\begin{corollary} \label{detect cor}
Let $U \subset I$ and $V \subset I$ be closed subsets properly contained in $I$.  The group $\langle G^U, G^V \rangle$ generated by $G^U$ and $G^V$ has nonabelian centralizer if and only if $U \cap V$ contains an open set.  
\end{corollary} 

\begin{proof} 
The condition that $U$ and $V$ are closed and properly contained in $I$ implies that $G^U$ and $G^V$ are nontrivial and in this case nonabelian.  If $\langle G^U, G^V \rangle$ has nonabelian centralizer, then by Proposition \ref{simpler prop}
$\fix(\langle G^U, G^V \rangle)$ contains an open subset of $I$.  This subset must be fixed pointwise by both $G^U$ and $G^V$, so is a subset of $U \cap V$.  

Conversely, if $U \cap V$ contains an open set $J$, then the group of all diffeomorphisms supported on $J$ commutes with $G^U$ and $G^V$, so commutes with $\langle G^U, G^V \rangle$.  
\end{proof}

Now we can begin the real work of the proof of Proposition \ref{flows}, using our continuity criterion.  Let $U \subset I$ be an open set bounded away from 0 and 1.
Since $\phi^t$ is continuous in $t$, there is some $\delta > 0$ such that for all $t<\delta$ we have $\phi^t(U) \cap U \neq \emptyset$.   By Corollary \ref{detect cor}, this is equivalent to the statement that $\langle G^U, G^{\phi^t(U)} \rangle$ has nonabelian centralizer.  
We want to translate this into a statement about $\Phi(G^U)$ and $\psi^{\alpha(t)}$.  

To do this, note first that 
\begin{align} \label{conj}
G^{\phi^t(U)} := &\{g \in G : g(x) = x \text{ for all } x \in \phi^t(U)\} \notag \\ 
= &\{\phi^t g \phi^{-t}  : g(x) = x \text{ for all } x \in U\}  \\
= &\phi^t(G^U) \phi^{-t}  \notag
\end{align}

and more generally, for any group of diffeomorphisms $H$ and any diffeomorphism $g$, we have
\begin{equation} \label{benson}
g \fix (H) = \fix(g H g^{-1})
\end{equation} 

Now consider $\Phi(G^U)$.  Since $G^U$ has nonabelian centralizer, $\Phi(G^U)$ has nonabelian centralizer as well.  By Proposition \ref{simpler prop}, $\fix \Phi(G^U)$ contains an open set.  
Using (\ref{conj}), we have 
\begin{equation*}
\Phi(G^{\phi^t(U)}) = \Phi(\phi^t)(\Phi(G^U) \Phi(\phi^{-t})= \psi^{\alpha(t)}(\Phi(G^U))\psi^{-\alpha(t)} 
\end{equation*}  

and letting $U'$ denote $\fix \Phi(G^U)$, it follows from (\ref{benson}) that 
\begin{equation}  \label{4}
\fix \Phi(G^{\phi^t(U)})) = \psi^{\alpha(t)}({U'}).
\end{equation}

For any $t < \delta$, consider the subgroup 
$$H_t:= \langle \Phi(G^U),\, \Phi(G^{\phi^t(U)}) \rangle = \langle \Phi(G^U),\, \psi^{\alpha(t)}(\Phi(G^U))\psi^{-\alpha(t)}) \rangle$$
Using (\ref{4}) we have
$$\fix(H_t) = \fix(\Phi(G^U)) \cap \fix(\psi^{\alpha(t)}(\Phi(G^U))\psi^{-\alpha(t)}) = U' \cap  \psi^{\alpha(t)}({U'})$$

 \begin{figure*}
   \labellist 
  \small\hair 2pt
   \pinlabel $U$ at 105 145 
   \pinlabel $\phi^t(U)$ at 160 145
   \pinlabel $V=\interior \left(\fix(\Phi((G^U))\right)$ at 205 095
   \pinlabel $\psi^{\alpha(t)}(V)$ at 245 035
   \endlabellist
  \centerline{
    \mbox{\includegraphics[width=4in]{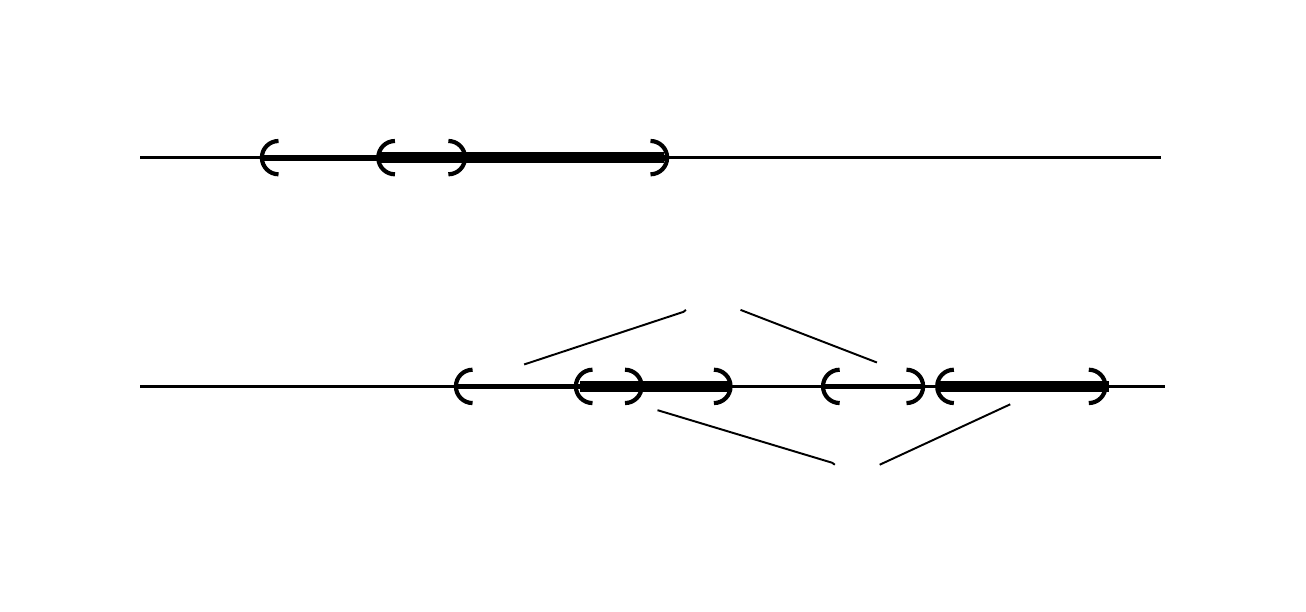}}}
 \caption{Continuity of $\alpha$ detected by intersections of open sets.}
  \label{contfig}
  \end{figure*}

$H_t$ also has nonabelian centralizer, so by Proposition \ref{simpler prop}, $\fix(H_t) = U' \cap  \psi^{\alpha(t)}({U'})$ contains an open set.   If $V = \interior(U')$,  we know that $V \cap \psi^{\alpha(t)}(V)$ contains an open set as well.   (A picture of the situation so far is given in Figure \ref{contfig}.)   If we knew further that $V$ was bounded away from $0$ and $1$, then we could apply Lemma \ref{detect} to $V$ and conclude that $\alpha$ is continuous.   

This would finish our proof.  Thus, it remains only to show that $V$ is bounded away from $0$ and $1$.  This is not hard to see, but uses a well-known theorem about diffeomorphism groups called the \emph{fragmentation property}. 

\begin{theorem}[\textbf{Fragmentation property}] \label{frag}
Let $M$ be a manifold and $\mathcal{B}$ an open cover of $M$.  Then the group $\Diff^r(M)_0$ is generated by the set $\{g \in \Diff^r(M)_0 : \supp(g) \in B \text{ for some } B \in \mathcal{B} \}$.  
\end{theorem} 
Recall that $\supp(g)$ is the \emph{support} of $g$, the closure of $M \setminus \fix(g)$.  A proof of Theorem \ref{frag} can be found in Chapter 2 of \cite{Ba}.  

Since $\{\phi^t\}$ is a flow, for large enough $s$ we have $\phi^s(\bar{U}) \cap \bar{U} = \emptyset$.  It follows from the fragmentation property that $G^U$ and $G^{\phi^s(U)}$ together generate $G$, so $\Phi(G^U)$ and $\Phi(G^{\phi^s(U)})$ generate $\Phi(G)$.  In particular, any point in $I$ that is fixed by both $\Phi(G^U)$ and $\Phi(G^{\phi^s(U)})$ is fixed by every element of $\Phi(G)$.  

To show now that $\interior(U')$ is bounded away from $0$ and $1$, assume for contradiction that $U': = \fix(\Phi(G^U))$ contains a neighborhood of $0$ (the case where it contains a neighborhood of $1$ is identical).  Then $\psi^{\alpha(s)}(U')$ contains a neighborhood of $0$ as well, but $\psi^{\alpha(s)}(U') = \fix \left(\Phi(G^{\phi^s(U)}) \right)$.  By the previous paragraph, we conclude that $\Phi(G)$ pointwise fixes a neighborhood of $0$, contradicting the fact that $\psi$ had no fixed points in $(0,1)$. This concludes the proof of Proposition \ref{flows}.  

\end{proof}

Our next goal is to remove the assumption that our flows had no fixed points and work with $\R$ instead of $I$.  There is some technical work to do here to keep track of supports, but the key ideas are really contained in the proof of Proposition \ref{flows}.   

\begin{proposition} \label{rflow}
Let $r \geq 3, p \geq2$ and
let $\Phi: \Diff_c^r(\R) \to \Diff_c^p(\R)$ be a homomorphism.  
For any $z \in \R$, there is an $n \in \N$ such that any flow $\{\phi^t\} \subset \Diff_c^r(\R)$ without fixed points in $(-n, n)$ has $\Phi(\{\phi_t\})$ continuous in $t$ on some neighborhood containing $z$.  
\end{proposition} 

\begin{remark} It will follow from Theorem \ref{main} that for any flow $\{\phi^t\}$, the image $\Phi(\{\phi_t\})$ will be continuous in $t$ everywhere, but we won't see this until we finish the whole proof of the theorem.
\end{remark}

\begin{proof}[Proof of Proposition \ref{rflow}]
Let $G$ denote $\Diff_c^r(\R)$.  We start by taking some steps to make the set-up in this situation as close as possible to the hypotheses of Proposition \ref{flows}.  

First, we may assume without loss of generality that $\fix(\Phi(G)) = \emptyset$.  This is because each connected component $C$ of $\R \setminus \fix(\Phi(G))$ is $\Phi(G)$ invariant, so we can consider the restriction of elements in the image of $\Phi$ to $C$, giving a homomorphism $\Diff_c^r(\R) \to \Diff_c^p(C) \cong \Diff_c^p(\R)$.  Since $G$ is perfect, this homomorphism is still injective.  

Let $\{\phi^t\} \subset G$ be a flow, and consider a connected component of $\R \setminus \fix(\{\phi^t\})$.  
Identify this interval with $(0,1)$ and consider its half-closure to be $[0,1) = I$.   As in the proof of Proposition \ref{flows}, let $U \subset I$ be an open interval bounded away from 0 and 1.  

Choose any $s \neq 0$ and let $\psi = \Phi(\phi^s)$.  Since $\psi$ is compactly supported, it has nonempty fixed set in $\R$.  Let $J$ be a connected component of $\R \setminus \fix(\psi)$.  Since $\phi^t$ and $\phi^s$ commute, $J$ is $\{\phi^t\}$-invariant and we can apply Szekeres' theorem to $\bar{J}$ and conclude as before that there is some flow $\{\psi^t\}$ on $J$ and an additive group homomorphism $\alpha_J: \R \to \R$ such that $\Phi(\phi_t)|_J = \psi^{\alpha_J(t)}$.  
If $\fix(\Phi(G^U)) \cap J$ contains an open set, then the argument from Proposition \ref{flows} applies verbatim to show that $\alpha_J$ is continuous, and so $\Phi(\{\phi^t\})$ is continuous on $J$.

Thus, the next lemma will finish the proof of Proposition \ref{rflow}.
\begin{lemma}\label{littlelem} 
Let $G = \Diff^r_c(\R)$ and suppose $\Phi: G \to \Diff^p_c(\R)$ satisfies $\fix(\Phi(G)) = \emptyset$. Then
\begin{enumerate}[(a)]
\item for any $z \in \R$, there is a bounded set $U \subset \R$ such that $z$ is in the interior of $\fix(\Phi(G^U))$, and
\item any flow $\{\phi^t\} \subset G$ with sufficiently large support has $z \notin \fix(\Phi(\phi^s))$ for some $s$.  
\end{enumerate}  
\end{lemma}
Given this lemma, we know that any flow with sufficiently large support will have a connected component $J$ of $\R \setminus \fix(\Phi(\phi^s))$ with $z$ contained inside an open set in $\fix(\Phi(G^U)) \cap J$, so the argument we just gave shows that $\Phi(\{\phi^t\})$ is continuous on $J$.

\bigskip
\noindent \textit{Proof of Lemma \ref{littlelem}}.
Let $z \in \R$.  Suppose for contradiction that there is no bounded, open interval $U$ such that $z \in \interior \left(\fix(\Phi(G^U))\right)$.  Let $G_U := \{g \in G: g(x) = x \text{ for all } x \notin U\}$.  By the fragmentation property (Theorem \ref{frag}), the subgroups $G_U$ generate $G$ as $U$ ranges over open intervals of $\R$.  

Now for any open set $U$, the subgroup $G_U$ commutes with $G^U$.  Also, $G_U$ is perfect and $G^U$ nonabelian, so by Corollary \ref{perfectcor}, $\R \setminus \fix(\Phi(G^U)) \subset \fix(\Phi(G_U))$.  It follows that $z \in \fix(\Phi(G_U))$ for every $U$.  Since the $\Phi(G_U)$ generate $\Phi(G)$, we have $z \in \fix(\Phi(G))$, a contradiction.   This proves (a).  

To prove (b), take a bounded set $U$ such that $z$ is in the interior of $\fix(\Phi(G^U))$.  
We claim that any flow $\{\phi^t\} \subset G$ with support containing $U$ satisfies $z \notin \fix(\Phi(\phi^s))$ for some $s$.  
To show this, suppose $\{\phi^t\}$ is such a flow.  We know that there is some $s$ such that $\phi^s(\bar{U}) \cap \bar{U} = \emptyset$, so $\phi^s$ and $G^U$ generate $G$ (using fragmentation again).  In particular, if $z \in \fix(\Phi(G^U)) \cap \fix(\Phi(\phi^s))$, then $z \in \fix(\Phi(G))$, which we assumed to be empty.  So $z \notin \fix(\Phi(\phi^s))$ and this is what we wanted to show.  

This completes the proof of the lemma and the proof of Proposition \ref{rflow}.  

\end{proof}

 \begin{figure*}
   \labellist 
  \small\hair 2pt
   \pinlabel $U$ at 100 060
   \pinlabel $\phi^t(U)$ at 143 060
   \pinlabel $z$ at 117 028
   \pinlabel $\Phi(\phi^s)(z)$ at 177 028
   \endlabellist
  \centerline{
    \mbox{\includegraphics[width=4in]{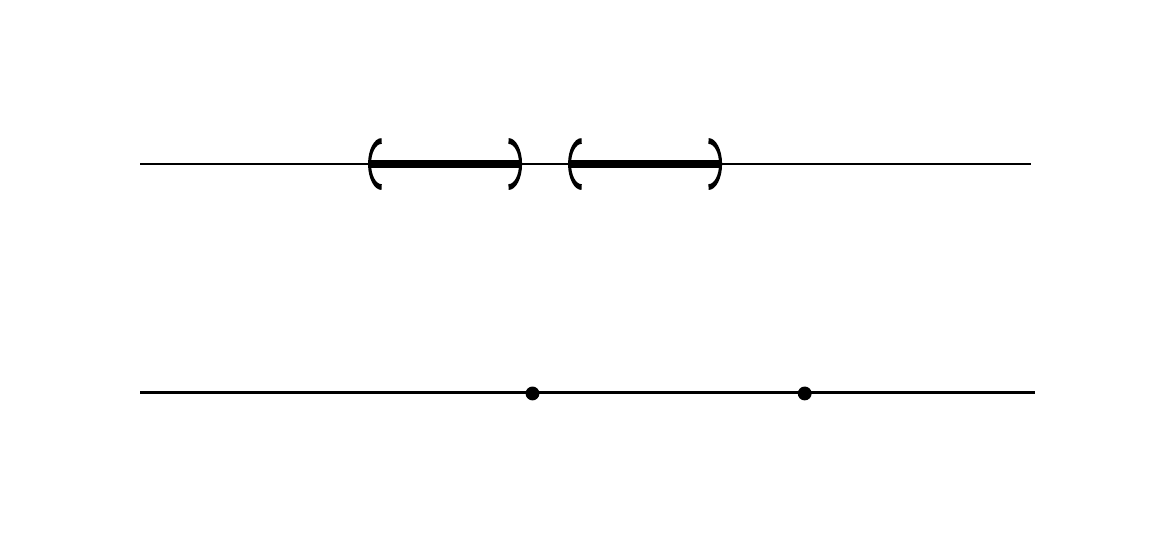}}}
 \caption{If $\phi^s(\bar{U}) \cap \bar{U} = \emptyset$, then $z$ is not fixed by $\Phi(\phi^s)$}.
  \label{contfig}
  \end{figure*}

To extend Proposition \ref{rflow} to homomorphisms $\Phi: \Diff^r_c(\R) \to \Diff^p_+(S^1)$, we need a lemma.  This is the diffeomorphism group analog to the fact that any injection $\R \hookrightarrow S^1$ must miss a point.  

\begin{lemma} \label{notempty}
Let  $r, p \geq 2$ and let $\Phi: \Diff^r_c(\R) \to \Diff^p_+(S^1)$ be a homomorphism.  Then $\fix \left(\Phi(\Diff^r_c(\R)) \right) \neq \emptyset$.
\end{lemma}

\begin{proof}
Let $G_n$ be the group of diffeomorphisms with support contained in $(-n, n)$.  Then $G_n \subset G_{n+1}$ and
$$G = \bigcup_{n \in \N} G_n.$$  
 It follows that $\Phi(G_n) \subset \Phi(G_{n+1})$, so 
\begin{equation}  \label{nested}
\fix(\Phi(G_n)) \supset \fix(\Phi(G_{n+1}))
\end{equation} 
  Moreover, since $$\Phi(G) = \bigcup_{n \in \N} \Phi(G_n)$$ we have also 
 \begin{equation}  \label{fills}
\fix(\Phi(G)) = \bigcap_{n \in \N} \fix(\Phi(G_n))
\end{equation}  

Now each $G_n$ commutes with the subgroup $H_n$ of diffeomorphisms fixing $[-n, n]$ pointwise and the subgroups $\Phi(G_n)$ and $\Phi(H_n)$ satisfy the hypotheses of Proposition \ref{keypropS1}.  In particular, if we let $G'_n$ be the group of elements of $\Phi(G_n)$ with rotation number zero, this is a nontrivial \emph{subgroup} (by remark \ref{rot0}) so by simplicity of $G_n$, we have $G'_n = G_n$. It now follows from Proposition \ref{keypropS1} that $\fix(\Phi(G_n)) \neq \emptyset$.  Since fixed sets are closed, it follows from (\ref{nested}) and (\ref{fills}) that $\fix(\Phi(G)) \neq \emptyset$.  
\end{proof}

The following corollary of Proposition \ref{rflow} is now immediate. 

\begin{corollary} \label{sflow} 
Let $r \geq 3$, $p \geq 2$, and suppose $\Phi: \Diff^r_c(\R) \to \Diff^p_+(S^1)$ is a homomorphism.  For any $z \in S^1 \setminus \fix(\Phi(G))$, there is an $n \in \N$ such that if $\{\phi^t\} \subset \Diff_c^r(\R)$ is a flow without fixed points in $(-n, n)$, then $\Phi(\{\phi_t\})$ is continuous in $t$ on a neighborhood containing $z$.    
\end{corollary}

\begin{proof} 
By Lemma \ref{notempty}, there is some point $x \in \fix \left( \Phi(\Diff^r_c(\R))\right)$.  
Moreover, the derivative of every element $\Phi(g) \in \Phi(\Diff^r_c(\R))$ at $x$ is zero, for $g \mapsto D\Phi(g)(x)$ is a homomorphism to $\R$, and since $\Diff^r_c(\R)$ is perfect, it must be trivial.  The same applies to higher order derivatives.  
Cutting $S^1$ at $x$, we can consider $\Phi$ to be a homomorphism into the subgroup of diffeomorphisms in $\Diff^p_+(S^1 \setminus\{x\})$ where all derivatives vanish at the endpoints.  This is isomorphic to a subgroup of $\Diff^p_c(\R)$ and  Proposition \ref{rflow} now applies.  

\end{proof}

A similar conclusion holds for homomorphisms $\Phi: \Diff^r_+(S^1) \to \Diff^r_+(S^1)$.  Since we will not use this in the sequel, we state it informally and only sketch the proof, leaving details to the reader.
\begin{corollary} \label{ssflow} 
Let $ r \geq 3$, and suppose $\Phi: \Diff^r_+(S^1) \to \Diff^r_+(S^1)$ is a homomorphism.  Then for any $z \in S^1$, any flow $\{\phi^t\} \subset \Diff_+^r(S^1)$ with sufficiently large support has $\Phi(\{\phi^t\})$ continuous at $z$
\end{corollary}

\begin{proof}[Proof sketch]  There are two classes of $\R$ subgroups of $\Diff_+(S^1)$, those with nonempty fix set and those conjugate to the rotation subgroup.  If $\fix(\{\phi^t\}) \neq \emptyset$, we can cut $S^1$ at a point $x \in \fix(\{\phi^t\})$ and apply Corollary \ref{sflow} with the group of diffeomorphisms of $S^1 \setminus \{x\} \cong \R$ as the domain.  Otherwise, $\{\phi^t\}$ is conjugate to the rotation subgroup, and we can use torsion elements and H\"older-Kopell theory to show that  $\Phi(\{\phi^t\})$ is conjugate to the rotation subgroup as well.  Continuity in $t$ can be demonstrated for small $t$ since for small $t$, we can choose $x$ and $y \in S^1$ and write rotation by $t$ as $f^t g^t$ where $f^t$ is a flow supported on $S^1 \setminus \{x\}$ and $g^t$ a flow supported on $S^1 \setminus \{y\}$.  We may also choose $f^t$ and $g^t$ to have arbitrarily large supports.  Corollary \ref{sflow} then implies that $\Phi(\{f^t\})$ and $\Phi(\{g^t\})$ are continuous at every point, hence so is $\Phi(\{\phi^t\})$ -- at least for small $t$.  But since $\Phi(\{\phi^t\})$ is an $\R$-subgroup, it must be continuous for all $t$.  

\end{proof}

\section{Proof of Theorem \ref{main}} \label{mainpf}

We start by proving Theorem \ref{main} for diffeomorphisms of $\R$.  Namely, we show the following.

\begin{theorem}  \label{r to r}
Let $r \geq 3$, $p \geq 2$, and let $\Phi: \Diff_c^r(\R) \to \Diff_c^p(\R)$ be nontrivial.  Then $\Phi$ is topologically diagonal.  If $r =s$, then $\Phi$ is  $C^r$-topologically diagonal.  
\end{theorem}

The proof of this theorem contains the essence of the proof of Theorem \ref{main}.  The only difference in the statement of Theorem \ref{main} is that $\Diff_c^r(\R)$ and/or $\Diff_c^p(\R)$ can be replaced with $\Diff_+^r(S^1)$ (and/or, $\Diff_+^p(S^1)$).  We deal with this easily at the end of this section.  

\begin{proof}[Proof of Theorem \ref{r to r}]
Let $G$ denote the group $\Diff_c^r(\R)$ and, as before, for a subset $U \subset \R$ let $G^U$ denote the subgroup $\{g \in G : gx = x \text{ for all } x \in U \}$. 

Let $\Phi: G \to G$ be a nontrivial homomorphism.  
Assume first for simplicity that $\fix(\Phi(G)) = \emptyset$.  In this case, we need only construct one map $f: \R \to \R$ such that $\Phi(g) = f g f^{-1}$ for all $g \in G$.  To construct $f$, we use Proposition \ref{rflow} (continuity of flows) along with Proposition \ref{simpler prop} to show that particular subgroups of point stabilizers in $G$ map under $\Phi$ to subgroups of point stabilizers.  This lets us build a point-to-point $\R \to \R$ map out of the data of $\Phi$.  

For $x \in \interior(I)$, let 
$$G^x  := \langle G^{(-\infty,x]},\, G^{[x,\infty)} \rangle.$$  
Since $G^{(-\infty,x]} $ and $G^{[x,\infty)}$ commute, we have 
\begin{equation} \label{product}
G^x \cong G^{(-\infty,x]} \times  G^{[x,\infty)} \end{equation}
Note that the set $G^x$ is contained in, but not equal to, the point stabilizer $G_x$ --  we will need to use the fact that $G^x$ decomposes as a direct product so the full point stabilizer will not work.   However, the fragmentation property still implies that 
\begin{equation}\label{generate eq}
\text{ if } g(x) \neq x, \text{ then } \langle G^x, g\rangle = G.
\end{equation}

Now let $A_x := \fix(\Phi(G^x))$.  The product structure of $G^x$ in (\ref{product}) means that  
$$\Phi(G^x) =  \Phi(G^{(-\infty,x]}) \times \Phi(G^{[x,\infty)}).$$

Corollary \ref{simpler prop cor} shows that two, nonabelian, commuting subgroups must have a common fixed point.  In particular, $\Phi(G^{(-\infty,x]})$ and $\Phi(G^{[x,\infty)})$ have a common fixed point, so $A_x = \fix(\Phi(G^x)) \neq \emptyset$.  Eventually we will see that $A_x$ consists of a \emph{single} fixed point and $f(x) = A_x$ will be our map.   

\begin{lemma}[Easy properties of $A_x$]\label{properties} 
Let $\Phi: G \to G$ be a homomorphism and assume that $\fix(\Phi(G)) = \emptyset$.  Let $A_x = \fix(\Phi(G^x))$ as above.  Then 
\begin{enumerate}
\item $A_x \neq \emptyset$
\item $A_x \cap A_y = \emptyset$ whenever $x \neq y$  
\item if $g \in G$ and $g(x) = y$, then $\Phi(g)(A_x) = A_y$
\end{enumerate}
\end{lemma}

\begin{proof}
1 follows from our discussion above.  To prove 2, note that by the fragmentation property (Theorem \ref{frag}) if $x \neq y$ then $G^x$ and $G^y$ together generate $G$.  Now if some point $z \in \R$ is fixed by both $\Phi(G^x)$ and $\Phi(G^y)$, then $z$ is fixed by every element of $\Phi(G) = \langle  \Phi(G^x), \Phi(G^y) \rangle$ contradicting our assumption that $\fix(\Phi(G)) = \emptyset$.  This shows that $A_x \cap A_y = \emptyset$.  

The proof of 3 is elementary:  If $g$ is a diffeomorphism with $g(x) = y$, then 
$$G^y = g G^x g^{-1}$$ 
and so 
\begin{equation}\label{eqAy}
A_y = \fix(\Phi(g G^x g^{-1})) = \Phi(g)(\fix(G^x)) = \Phi(g)(A_x).
\end{equation}
\end{proof}

Using Lemma \ref{properties}, we now show that $A_x$ consists of isolated points.  Let $z \in A_x$ and let $\{\phi^t\} \subset G$ be a flow such that $\phi^t(x) \neq x$ for $t \neq 0$, and such that $\Phi(\phi^t)$ is continuous in $t$ at $z$.  Such a flow exists by Proposition \ref{rflow}, in fact any flow $\phi^t$ with large enough support will work.  Let $J$ be the connected component of $\R \setminus \fix(\Phi(\phi^t))$ containing $z$.  Proposition \ref{rflow} says that there is a flow $\psi^t$ on $J$ such that  $\Phi(\phi^t)|_J = \psi^{\lambda t}$.  
If $\psi^{\lambda t}(z) = z$ for $t \neq 0$, then $z$ is fixed by both $\psi^{\lambda t}$ and $\Phi(G^x)$.  It follows from (\ref{generate eq}) above that $z$ is then fixed by all of $\Phi(G)$.  This contradicts our assumption.  Thus, $z$ is not fixed by any such flow.  

This shows that $z$ must be an isolated point of $A_x$: if not, there is some $z' \in A_x \cap J$ and some small $t \neq 0$ such that $\psi^{\lambda t}(z) = z'$.  But Property 3 of Lemma \ref{properties} says that $z' \in A_{\phi^t(x)}$, and Property 1 says that $A_{\phi^t(x)} \cap A_x = \emptyset$.  

 \begin{figure*}
   \labellist 
  \small\hair 2pt
   \pinlabel $x$ at 85 045 
   \pinlabel $y$ at 150 045
   \pinlabel $\phi^t$ at 125 060
  \pinlabel $z$ at 115 008
   \pinlabel $f(y)$ at 175 008
   \pinlabel $\Phi(\phi^t)$ at 148 022
         \endlabellist
  \centerline{
    \mbox{\includegraphics[width=4in]{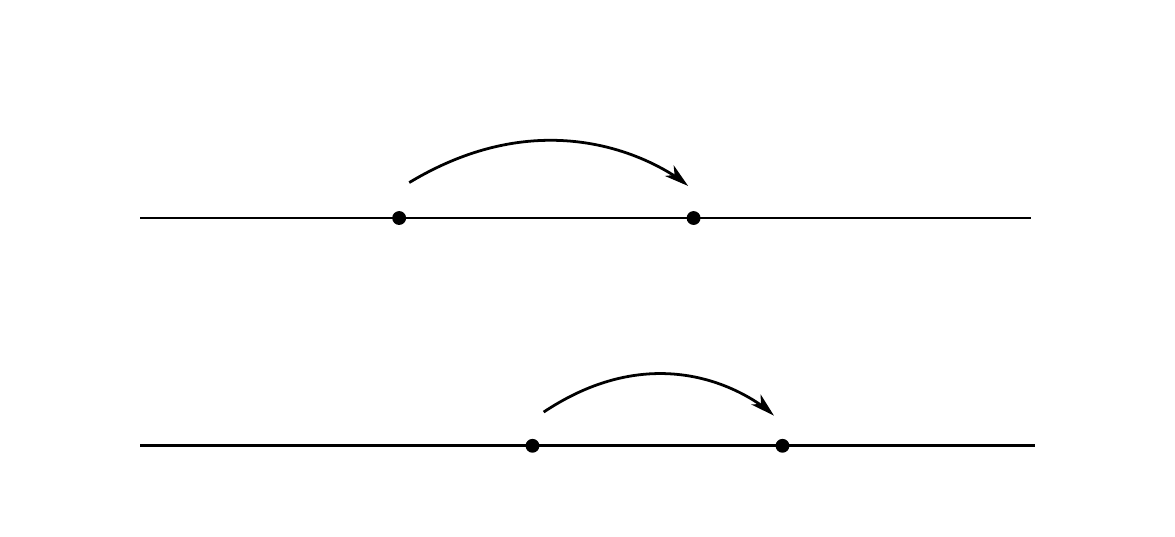}}}
 \caption{The definition of $f: \R \to \R$}
  \label{definingf}
  \end{figure*}

We now build a continuous map $f: \R \to \R$.   For $y \in \R$ define $f(y)$ as follows.  Take any flow $\{\phi^t\} \subset G$ such that $\phi^s(x) = y$ for some $s$ and such that $\Phi(\phi^t)$ is continuous in $t$ at $z$.  Then set $f(y) = \Phi(\phi^s)(z)$.  See Figure \ref{definingf}.   We will show that $f$ is 
\begin{enumerate}[(a)]
\item well-defined
\item continuous and injective
\item $G$-equivariant, meaning that $fg(y) = \Phi(g)(f(y))$ for all $g \in G$ and $y \in \R$
\item surjective 
\item $C^r$ in the case where $r = p$.  
\end{enumerate} 
These will all follow from the continuity of flows at $z$ and the easy properties of the sets $A_x$.  In particular, we will use the fact that for any flow $\{ \phi^t \}$ used in the definition of $f(y)$, we have
\begin{equation}
\Phi(\phi^s)(z) \in \Phi(\phi^s)(A_x) = A_y
\end{equation}

\noindent \textit{Well defined}.
Let $\{\phi^t\}$ and $\{\eta^t\}$ be two flows such that $\phi^s(x) = \eta^r(x) = y$ and such that both $\Phi(\phi^t)$ and $\Phi(\eta^t)$ are continuous in $t$ at $z$.  Suppose for contradiction that $\Phi(\phi^s)(z) \neq \Phi(\eta^r)(z)$.  Then $\Phi(\phi^s)(z) = \Phi(\eta^q)(z)$ for some $q \neq r$.  But then $\eta^q(x) = y'$ for some $y' \neq y$.  It follows that $\Phi(\eta^q)(z) \in A_{y'} \cap A_y$, a contradiction.  

\bigskip
\noindent \textit{Continuity and injectivity}.
Continuity is immediate from the definition of $f$ since $\Phi(\phi^t)$ is continuous on the connected component of its support that contains $z$.  Injectivity follows from continuity and the fact that $A_{y'} \cap A_y = \emptyset$ for $y \neq y'$.

\bigskip
\noindent \textit{$G$-equivariance}.
By construction, the image of $f$ contains exactly one point of $A_y$ for each $y \in \R$.  
Given $y \in \R$ and $g \in G$, let $\{\phi^t\}$ be a flow such that $\phi^1(x) = y$ and $\phi^s(x) = g(y)$ for some $s$.  We may also choose $\{\phi^t\}$ to have as large support as we want, so as to ensure it is continuous at $z$.  Then 
$$f(y) = \Phi(\phi^1)(z)$$ and 
$$\Phi(g)f(y) = \Phi(g) \Phi(\phi^1)(z) = \Phi(g \phi^1) (z)  \in A_{g(y)}.$$ 
We also have 
$$f(g(y)) = f(\phi^s(x)) = \Phi(\phi^s)(z) \in A_{g(y)}$$
so it follows that $fg(y) = \Phi(g)(f(y))$.  

\bigskip
\noindent \textit{Surjectivity}.
The $G$-equivariance of $f$ implies that the image of $f$ is $\Phi(G)$-invariant.  Since we assumed that $\fix(\Phi(G)) = \emptyset$, the image of $f$ must be $\R$.  

\bigskip
It remains to show that $f$ is $C^r$ under the assumption that $p = r$.  To do this, we need a generalization of the following theorem of Takens.

\begin{theorem}[Takens \cite{Ta}] \label{takens}
Let $f: M_1 \to M_2$ be a bijection between two smooth manifolds $M_1$ and $M_2$ with the property that $g: M_1 \to M_2$ is a $C^\infty$ diffeomorphism if and only if $fgf^{-1}$ is a $C^\infty$ diffeomorphism.  Then $f$ is also a $C^\infty$ diffeomorphism. 
\end{theorem}

In our case, $f: \R \to \R$ satisfies the property that for each $g \in G$ the conjugate $fg f^{-1} = \Phi(g)$ is $C^r$.  We also know that $\Phi(G)$ acts transitively on the image of $f$.  Takens' original proof can be easily adapted to this case to show that $f$ is of class $C^r$.  We omit the details.  The reader may also consult the main theorem of Rybicki in \cite{Ry} for a similar (but more complicated to state) theorem that applies directly to our situation.  

Thus, we have shown that Theorem \ref{r to r} holds provided that $\fix(\Phi(G)) = \emptyset$.  
If instead $\fix(\Phi(G)) \neq \emptyset$, we can pick any connected component $J$ of $\R \setminus \fix(\Phi(G))$ and run the proof above for the induced group homomorphism $\Phi_J: G \to \Diff^r(J)$  given by $\Phi_J(g) = \Phi(g)|_J$.  What we get is a homeomorphism (or $C^r$-diffeomorphism in the $p=r$ case) $f: \R \to J$ such that  $f (g(x)) = \Phi(g)(f(x))$ for all $g \in G$.  Doing this separately on each connected component defines the topologically diagonal embedding.  

This completes the proof of Theorem \ref{r to r}

\end{proof}

Using that the image of our homomorphism is $\Diff^p_c(\R)$, i.e. all diffeomorphisms in the image have compact support, we may conclude something stronger about the topologically diagonal embedding.

\begin{proposition} \label{finitelymany}
Suppose $p, r \geq 2$ and $\Phi: \Diff^r_c(\R) \to \Diff^p_c(\R)$ is a topologically diagonal embedding.  Then 
there are only finitely finitely many connected components of $\R \setminus \fix \Phi(\Diff^r_c(\R))$.  
\end{proposition} 

\begin{proof}
This will follow from the fact that $\Phi(G)$ acts by $C^2$ diffeomorphisms along with the ``diagonallity" of the action of $\Phi(G)$.  That the embedding is topologically diagonal means exactly that for any two connected components $J_1$ and $J_2$ of $\R \setminus \fix \left( \Phi(G) \right)$ with embeddings $f_i: \R \to J_i$, we have 
$$(f_2 f_1^{-1}) \circ \Phi(g)|_{J_1} \circ (f_1f_2^{-1})  = \Phi(g)|_{J_2}$$
for each $g \in G$.  
In other words, $f_2f_1^{-1}$ conjugates the action of $G$ on $J_1$ with the action of $G$ on $J_2$.  
We also know that there is some \emph{compact} subset of $\R$ that contains all the connected components $J_i$ that are bounded.   Moreover, Taken's theorem (Theorem \ref{takens} below) implies that each $f_2f_1^{-1}$ is of class $C^s$ where $s = \min(p, r) \geq 2$.  One can show that such an action can be made $C^1$, but it will necessarily have unbounded derivatives.   See e.g. \cite{Na} for more details.  
\end{proof}

\boldhead{Finishing the proof of Theorem \ref{main}}
We conclude by stating the necessary modifications to replace either copy of $\R$ with $S^1$ in the statement of Theorem \ref{r to r} and therefore prove Theorem \ref{main}.  There are three cases to consider.  

\bigskip \noindent 1. Homomorphisms $\Phi: \Diff^r_+(S^1) \to \Diff^p_c(\R)$
 
There are no nontrivial such homomorphisms since $\Diff^r_+(S^1)$ is simple and has torsion. 

\bigskip \noindent 2. Homomorphisms $\Phi: \Diff^r_c(\R) \to \Diff^p_+(S^1)$

For this, the proof of Theorem \ref{r to r} above works verbatim and shows that this embedding is topologically diagonal.  Note in particular that there must be some point in $S^1$ fixed by every diffeomorphism in the image of $\Phi$

\bigskip \noindent 3. Homomorphisms $\Phi: \Diff^r_+(S^1) \to \Diff^p_+(S^1)$

$\Diff^r_+(S^1)$ contains many copies of $\Diff^r_c(\R)$ (as in Example \ref{easy} c) since there are many embeddings of $\R$ in $S^1$.  For any such copy, consider the restriction of $\Phi$ to the $\Diff^r_c(\R)$ subgroup, a homomorphism $\Diff^r_c(\R) \to \Diff^p_+(S^1)$.  We know already that this must be topologically diagonal.  
The fact that this is true for \emph{any} copy means that the embeddings are compatible, in the sense that if $f_i$ and $f_j$ are any two maps $\R \to S^1$ coming from any two embeddings of copies of $\Diff^r_c(\R)$, then $f_i$ and $f_j$ agree on their common domain.  It follows that the $f_i$ patch together to define a globally defined, continuous map $f: S^1 \to S^1$, which is locally a homeomorphism.   The fact that $f$ is equivariant with respect to the action of $\Diff^r_+(S^1)$ and that $\Diff^r_+(S^1)$ contains finite order rotations means that $f$ must be a global homeomorphism.  Again, we can use Takens' theorem to show that $f$ is $C^r$ provided that $p=r$.  

It is also possible to show that $\Phi$ is topologically diagonal in this case by using flows on $S^1$ and Corollary \ref{ssflow}.

\section{Application: dimension as an algebraic invariant} \label{ghys}

Recall the statement of Theorem \ref{main cor}
\begin{dimthm}Let $r \geq 3$, $p \geq 2$, let $M_1$ be any manifold and $M_2$ a 1-manifold.  Suppose that $\Phi: \Diff_c^r(M_1) \to \Diff_c^p(M_2)$ is an injective homomorphism.  Then $\dim(M_1) = 1$ and $\Phi$ is topologically diagonal  If $r=p$, then $\Phi$ is $C^r$-topologically diagonal. 
\end{dimthm}

To prove this it will be enough to show that $\dim(M_1) = 1$, since in that case it follows from Theorem \ref{main} that $\Phi$ is topologically diagonal.  

We start with an easy corollary of Szekeres's theorem.
\begin{corollary}\label{sz cor}  
Let $r \geq 2$ and let $g$ be a $C^r$-diffeomorphism of $I = [0,1)$ with no interior fixed points.  Then $g^2$ and $g$ have the same centralizer in $\Diff^r_+(I)$.  
\end{corollary}

\begin{proof}
By uniqueness of the flow in Szekeres' theorem, $C_{\Diff^1_+(I)}(g) = C_{\Diff^1_+(I)}(g^2)$.  Since 
$$C_{\Diff^r_+(I)}(g) = C_{\Diff^1_+(I)}(g) \cap \Diff^r(I)$$ 
we have also 
$$C_{\Diff^r_+(I)}(g) = C_{\Diff^r_+(I)}(g^2).$$  

\end{proof}

With a little more work, we can show that Corollary \ref{sz cor} is also true for compactly supported diffeomorphisms of $\R$.

\begin{lemma}\label{cent lemma} 
Let $r \geq 2$ and let $g \in \Diff^r_c(\R)$.  Then $C_{\Diff^r_c(\R)}(g) = C_{\Diff^r_c(\R)}(g^2)$ 
\end{lemma}

\begin{proof}  
That $C(g) \subset C(g^2)$ is immediate.  We prove the reverse containment also holds.  

Let $g \in \Diff^r_c(\R)$ and suppose $r \geq 2$.  Then $\fix(g) \neq \emptyset$ and for any connected component $J$ of $\R \setminus \fix(g)$ we can apply Szekeres' lemma to $g|_{\bar{J}}$.  Note that since $g$ is orientation-preserving, $\fix(g) = \fix(g^2)$.  

Consider a connected component $J$ of $\R \setminus \fix(g)$.  By Corollary \ref{sz cor}, any diffeomorphism that leaves $J$ invariant and commutes with $g^2$ on $J$ also commutes with $g$ on $J$.  A $C^r$-diffeomorphism $h$ that commutes with $g^2$ but does \emph{not} leave $J$ invariant must map $J$ to another connected component of $\R \setminus (\fix(g))$, say $J'$.  Moreover, $h$ must conjugate the action of $g^2$ on $J$ to the action of $g^2$ on $J'$.  By uniqueness of the flow from Szerkes' theorem, if $g^t_J$ is the flow on $J$ that commutes with $g^2|_J$ then $h g^t_J h^{-1}$ is the flow on $J'$ that commutes with $g^2|_{J'}$.  It follows that $h$ conjugates the action of $g = (g^2)^{\frac{1}{2}}$ on $J$ to the action of $g$ on $J'$.   In other words, for any $x \in J$, we have $hg(x) = gh(x)$, i.e. $h$ commutes with $g$ on $J$.   Since this is true for any component of $\R \setminus \fix(g)$, it follows that $h$ commutes with $g$.  
 
 \end{proof}

Now we can easily prove Theorem \ref{main cor} if $M_2 = \R$.
\begin{proof}

Suppose that $M_1$ is a manifold of dimension $n \geq 2$ and that $\Phi$ is an injective homomorphism $\Diff_c^r(M_1)_0 \to \Diff_c^p(\R)$.  Let $B^n$ be an open ball in $M_1$ and consider the induced injective homomorphism $\Diff_c^r(B^n) \to \Diff_c^r(\R)$ given by restricting the domain of $\Phi$ to elements supported in $B^n$.   We claim that no such homomorphism exists.  
 
To see this, we will construct a diffeomorphism $g \in \Diff_c^r(B^n)$ such that 
$$C(g^2)_{\Diff_c^r(B^n)} \neq C(g)_{\Diff_c^r(B^n)},$$ 
i.e. such that there is some $h \in \Diff_c^r(B^n) \subset \Diff_c^r(M_1)$ where $[h, g^2] = 1$ but $[h, g] \neq 1$.  In this case we will have that $[\Phi(h), \Phi(g^2)] = 1$ but $[\Phi(h), \Phi(g)] \neq 1$ in $\Diff_c^r(\R)$, contradicting Lemma \ref{cent lemma}.  

To construct $g$, choose any diffeomorphism that acts as rotation by $\pi$ on a small cylinder $B^2 \times [0,1]^{n-2}$ inside $B^n$, and has trivial centralizer outside $D^2 \times [0,1]^{n-2}$.  Precisely, we want the restriction of $g$ to $B^2 \times [0,1]^{n-2}$ to be given by $(r e^{i \theta},\, x) \mapsto (r e^{i (\pi + \theta)},\, x)$. 
Now let $h$ be any diffeomorphism supported inside $D^2 \times [0,1]^{n-2}$ that does not commute with $g$.  By construction, the restriction of $g^2$ to $D^2 \times [0,1]^{n-2}$ is trivial, so $h$ commutes with $g^2$, as desired.  

\end{proof}

Unfortunately, there \emph{are} elements $g \in \Diff^r(S^1)$ such that $C(g) \neq C(g^2)$ -- finite order rotations are one example, but there are infinite order such elements as well.  Thus, our proof above won't work if $M_2 = S^1$.  
To prove Theorem \ref{main cor} when $M_2 = S^1$, we'll use the description of $\Diff^r_c(\R)$ actions from Theorem \ref{main}.  

\begin{proof}[Proof of Theorem \ref{main cor}, $M_2 = S^1$ case]
 
Suppose that $M_1$ is a manifold of dimension $n \geq 2$ and that $\Phi$ is an injective homomorphism $\Diff_c^r(M_1)_0 \to \Diff_c^p(S^1)$. We want to produce a contradiction.  
We claim that it is enough to show that there is no injective homomorphism $\Diff^r_c(B^2) \to \Diff^p_+(S^1)$. (Recall $B^2$ is the open 2-dimensional disc).  This is because of the following easy lemma.

\begin{lemma} \label{embed lemma}
$\Diff^r_c(B^2)$ embeds as a subgroup of $\Diff^r_c(M_1)$ for any manifold $M_1$ of dimension at least two.
\end{lemma}

\begin{proof}[Proof of Lemma \ref{embed lemma}]
This is trivial if $\dim(M_1) = 2$, since $B^2$ embeds in $M_1$ as an open submanifold.  If $\dim(M_1) = n >2$, we can embedd $\Diff^r_c(B^2)$ in $\Diff^r_c(B^2 \times S^{n-2})$ via $g \cdot (b, s) = (g(b), s)$.   Then take any embedding of $B^2 \times S^{n-2}$ in $M_1$ as an open submanifold.  This defines an embedding $\Diff^r_c(B^2 \times S^{n-2}) \hookrightarrow \Diff^r_c(M_1)$.
\end{proof}

Thus, we have reduced the proof of Theorem \ref{main cor} to the following: 

\begin{lemma} \label{nohomo}
There is no injective homomorphism $\Diff^r_c(B^2) \to \Diff^p_+(S^1)$
\end{lemma}

\noindent \textit{Proof of Lemma \ref{nohomo}}. 
First, we show that $\Diff_c^r(B^2)$ has a subgroup $H$ isomorphic to $\Diff_c^r(\R)$.  To do this, let $\tau: \Diff_c^r(\R) \to \Diff_c^r(\R \times S^1)$ be the map given by $\tau(g)(x, s) \mapsto (g(x), s)$, and let $f: (\R \times S^1) \to B^2$ be a $C^r$ embedding.   Then $f_{*} \tau: \Diff_c^r(\R) \to \Diff_c^r(B^2)$ is injective with image $H \cong \Diff_c^r(\R)$.  
See Figure \ref{figRinB}.

 \begin{figure*}
   \labellist 
  \small\hair 2pt
   \pinlabel $h$ at 85 164 
         \endlabellist
  \centerline{
    \mbox{\includegraphics[width=2.5in]{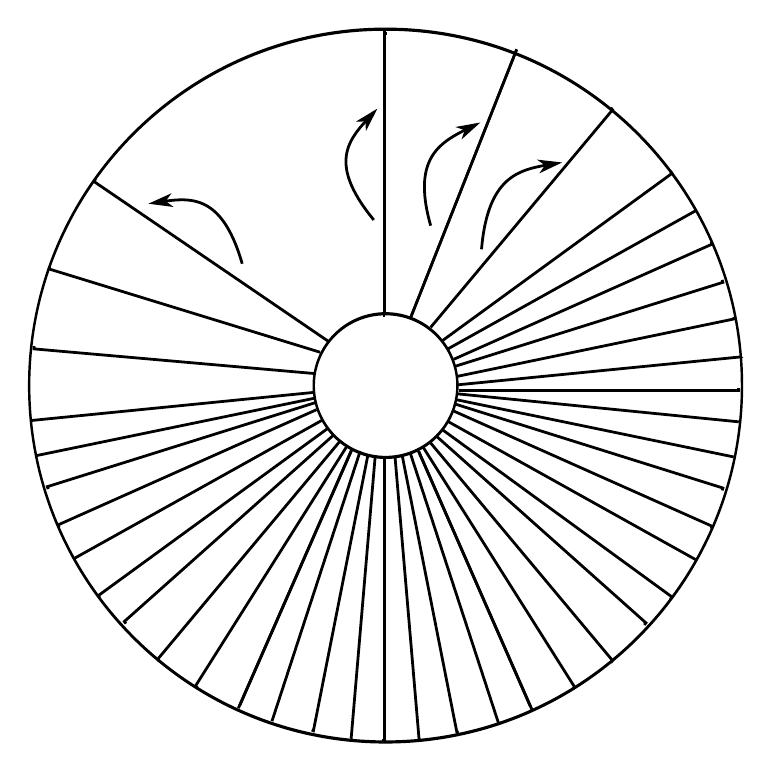}}}
 \caption{$\Diff_c^r(\R)$ embedds in $\Diff_c^r(B)$ by acting on the radii of an annulus}
  \label{figRinB}
  \end{figure*}

Suppose for contradiction that $\Phi: \Diff^r_c(B^2) \to \Diff^p_+(S^1)$ is injective.  
By Theorem \ref{main}, $\Phi(H)$ is a topologically diagonal embedding, and by Proposition \ref{finitelymany}, $\fix(\Phi(H))$ has finitely many connected components.  Let $N$ be the number of components of $\fix(\Phi(H))$.  Note that if $\phi$ is any diffeomorphism of $S^1$ that permutes the components of $S^1 \setminus \fix(\Phi(H))$, then $\phi^{N!}$ leaves each component of $S^1 \setminus \fix(\Phi(H))$ invariant.  

We will find a diffeomorphism $g \in \Diff_c^r(B^2)$ with support disjoint from $\supp(H)$, and such that the centralizer of $g^{2(N!)}$ is not equal to the centralizer of $g^{N!}$.  To construct such a $g$, we may use a diffeomorphism that acts as rotation by $\pi/N!$ on a small cylinder disjoint from $\supp(H)$, similar to the construction in the $M_2 = \R$ case above.  Let $g'$ be a diffeomorphism supported on the cylinder that does not commute with $g^{N!}$.  As in our construction for the $M_2 = \R$ case, $g'$ commutes with $H$ and with $g^{2(N!)}$ but not with $g^{N!}$.  

Since $g$ and $H$ commute, $\Phi(g)$ commutes with $\Phi(H)$, and $\Phi(g)^{N!}$ leaves invariant each subinterval or isolated point of $\fix(\Phi(H))$.  Moreover, since $\Phi(g)^{N!}$ commutes with $\Phi(H)$ and $H$ is nonabelian, $\Phi(g)^{N!}$ must be supported on the subintervals fixed by $\Phi(H)$ (we are using the proof of Proposition \ref{simpler prop} again here).  In particular, we can regard $\Phi(g)^{N!}$ and $\Phi(g)^{2N!}$ as diffeomorphisms of an interval, both with the same support.  

Finally, consider $\Phi(g')$.  Since $\Phi(g')$ commutes with $\Phi(H)$, we also have that  $\Phi(g')^{N!}$ leaves invariant each isolated point or subinterval of $\fix(\Phi(H))$.  In particular, it fixes some point and we may regard $\Phi(g')^{N!}$
as a diffeomorphism of the interval.  Then $\Phi(g')^{N!}$ commutes with $\Phi(g)^{2N!}$ but not with $\Phi(g)^{N!}$, contradicting Lemma \ref{cent lemma}.   

This completes the proof of Lemma \ref{nohomo} and the proof of Theorem \ref{main cor}.  

\end{proof}

\section{Compact support, simplicity, and smoothness} \label{SSS}

We remark on some of the hypotheses of our theorems and ask about generalizations.  

\boldhead{Compact support}
As mentioned in the introduction, the group $\Diff^r_c(M)$ is simple for $r \neq \dim(M)+1$.  
We believe that it should be possible to prove an analog of Theorem \ref{main} for $\Diff^r(\R)_0$ rather than $\Diff^r_c(\R)$ using the following theorem of Schweitzer, which classifies all normal subgroups of $\Diff^r(\R)_0$.  

\begin{theorem}[Schweitzer, \cite{Sc}]
Let $r \geq 3$.  There are only three proper normal subgroups of $\Diff^r(\R)_0$
\begin{enumerate}[(a)] 
\item The group of diffeomorphisms coinciding with the identity in a neighborhood of $\infty$
\item Those coinciding with the identity in a neighborhood of $-\infty$
\item  $\Diff^r_c(\R)$.
\end{enumerate}
\end{theorem} 

\begin{remark} This is actually a special case of a much more general theorem which classifies normal subgroups of $\Diff^r(\R^n)_0$, for any $r \geq 1$, such that $r \neq n+1$.  
\end{remark}

Thus, any non-injective homomorphism $\Phi: \Diff^r(\R)_0 \to \Diff^p(\R)_0$ has kernel equal to one of the subgroups above.  We conjecture that no such homomorphism exists.  
\medskip 

Note also that we did not make very heavy use of the fact that the \emph{target} group was a group of compactly supported diffeomorphisms (or even orientation-preserving diffeomorphisms).  It may be possible to adapt our proofs for homomorphisms $\Phi: \Diff^r_c(\R) \to \Diff^r(\R)$, starting with a modified version of Proposition \ref{simpler prop}.

\boldhead{Manifolds with boundary} 
Our results are not true for manifolds with boundary.  For example, $\Diff^r[0,1]$ admits nontrivial homomorphisms to $\R$ given by taking the derivative at $0$ or at $1$. These can be used to build homomorphisms of $\Diff^r[0,1]$ into $\Diff^r[0,1]$ or $\Diff^r_c(\R)$that are not topologically diagonal.  

However, I do not know whether it is possible to construct any counterexamples without using derivatives: 
\begin{question}  
Let $G^{\infty}$ be the group of $C^{\infty}$ diffeomorphisms of the interval that are infinitely tangent to the identity at 0 and at 1.  Do there exist homomorphisms $\Phi: G^{\infty} \to \Diff^r_c(\R)$ that are not topologically diagonal? 
\end{question}

\boldhead{Smoothness} 
There are also counterexamples to some of our tools when we remove assumptions on smoothness.  In Section \ref{comm} on commuting subgroups, we assumed all diffeomorphisms were of class at least $C^2$.  Although $C^{1+bv}$ would have sufficed, Kopell's Lemma and Denjoy's theorem do not hold for $C^1$ diffeomorphisms.  Namely, there are fixed point free $C^1$ diffeomorphisms of the interval commuting with $C^1$ diffeomorphisms that have fixed points, as well as $C^1$ diffeomorphisms of the circle with irrational rotation number that are not conjugate to irrational rotations (see \cite{Na}).  
Note also that the group $\Diff^{1+bv}(S^1)$ is known to be \emph{not} simple (see \cite{Ma3}) and whether the group $\Diff^2(S^1)$ is simple is unknown!  In particular, $\Diff^{1+bv}(S^1)$ admits non topologically diagonal and even non-continuous homomorphisms to other groups of diffeomorphisms, and it is possible that $\Diff^2(S^1)$ does as well.  

However, working in the category of \emph{homeomorphisms} rather than diffoemorphisms, E. Militon has recently obtained results in the same sprit as ours, using different techniques.  See \cite{Militon 1} and \cite{Militon 2}.  

\boldhead{Higher dimensions}
We conclude with a final remark on one of the difficulties in extending the results of this paper to manifolds of dimension $n>1$.  Our methods here relied heavily on properties of commuting subgroups, namely, the H\"older-Kopell theorey and Proposition \ref{simpler prop}.  The following example shows that Proposition \ref{simpler prop} does not hold in higher dimensions:

\begin{example}[Large centralizers in higher dimensions]
Consider the following two subgroups of $\Diff^r(\R^2)$. 
$$G:= \{f \in \Diff^r(\R^2) : f(x,y) = (f_1(x), y) \text{ for some } f_1 \in \Diff^r(\R) \}$$ 
$$H:= \{f \in \Diff^r(\R^2) : f(x,y) = (x, f_2(y)) \text{ for some } f_2 \in \Diff^r(\R) \}$$
These commute, i.e. $H$ is a subset of the centralizer of $G$, they are both ``large" in the sense that they are infinite dimensional, non-abelian, etc. and yet neither one fixes any nonempty subset of $\R^2$.  
\end{example}

Similar subgroups can be found in $\Diff^r(\T^2)$, in $\Diff^r(\R^n)$ and $\Diff^r(\T^n)$ for the $n$-dimensional torus.  However, we note that the groups $G$ and $H$ above are \emph{not} groups of compactly supported diffeomorphisms.  

\begin{question} Do there exist examples of large (for any reasonable definition of ``large") commuting subgroups in $\Diff^r_c(\R^n)$ that do not fix any open set?  Do there exist examples in $\Diff^r_c(M^n)$ for all $n$-manifolds?
\end{question}

A negative answer to this question would be the first step towards finding an algebraic-topological correspondence for manifolds of higher dimension, analogous to the correspondence we found here for 1-manifolds.


Dept. of Mathematics 

University of Chicago 

5734 University Ave. Chicago, IL 60637 

E-mail: mann@math.uchicago.edu

\newpage

\end{document}